\documentclass[11pt]{amsart}
\usepackage{amssymb,amsmath,amsthm, verbatim}
\usepackage{fullpage}
\usepackage{xcolor}
\usepackage{url}
\usepackage{mathrsfs}
\usepackage[all]{xy}
  \SelectTips{cm}{10}
  \everyxy={<2.5em,0em>:}
\usepackage{tikz}
\usepackage{picinpar} 

\usepackage{ifpdf}
  \ifpdf
    \usepackage{hyperref} 
  \else
    
    \newcommand{\href}[2]{#2}
  \fi

\theoremstyle{plain}
  \newtheorem{lemma}[equation]{Lemma}
  \newtheorem{proposition}[equation]{Proposition}
  \newtheorem{theorem}[equation]{Theorem}
  \newtheorem{corollary}[equation]{Corollary}
    \newtheorem{question}[equation]{Question}

    \newtheorem{statement}[equation]{Statement}

\theoremstyle{definition}
  \newtheorem{definition}[equation]{Definition}

\theoremstyle{remark}
  \newtheorem{remark}[equation]{Remark}

\renewcommand{\thesection}{\arabic{section}}
\renewcommand{\theequation}{\thesection.\arabic{equation}}

 \DeclareFontFamily{U}{manual}{}
 \DeclareFontShape{U}{manual}{m}{n}{ <->  manfnt }{}
 \newcommand{\manfntsymbol}[1]{%
    {\fontencoding{U}\fontfamily{manual}\selectfont\symbol{#1}}}

\makeatletter
   \@addtoreset{section}{part}
   \@addtoreset{equation}{section}
   \@addtoreset{footnote}{section}

    {\hspace*{\fill}$\lrcorner$\endgraf\endgroup\end{trivlist}}
 \newenvironment{example}[1][]{
   \refstepcounter{equation}
   \begin{proof}[Example~\theequation%
   \@ifnotempty{#1}{ (#1)}.]
   }
  {\end{proof}}
\makeatother

  \DeclareFontFamily{OT1}{pzc}{}
  \DeclareFontShape{OT1}{pzc}{m}{it}{<-> s * [1.100] pzcmi7t}{}
  \DeclareMathAlphabet{\mathpzc}{OT1}{pzc}{m}{it}

\newif\ifhascomments \hascommentstrue
\ifhascomments
  \newcommand{\jenna}[1]{{\color{red}[[\ensuremath{\bigstar\bigstar\bigstar} #1]]}}
  \newcommand{\matt}[1]{{\color{red}[[\ensuremath{\spadesuit\spadesuit\spadesuit} #1]]}}
\else
  \newcommand{\jenna}[1]{}
  \newcommand{\matt}[1]{}
\fi

\renewcommand{\>}{\rangle} 

\renewcommand{\AA}{\mathbb{A}}

\DeclareMathOperator{\Aut}{\ensuremath{\mathcal{A}\kern-.125em\mathpzc{ut}}}

\DeclareMathOperator{\ext}{Ext}

\DeclareMathOperator{\Endo}{\ensuremath{\mathcal{E}\kern-.125em\mathpzc{nd}}}

\let\hom\relax
\DeclareMathOperator{\hom}{Hom}

\DeclareMathOperator{\Hom}{\ensuremath{\mathcal{H}\kern-.125em\mathpzc{om}}}

\newcommand{\m}{\mathfrak m}

\renewcommand{\setminus}{\smallsetminus}

 \def\ari[#1]{\ar@{^(->}[#1]}
 \def\are[#1]{\ar[#1]^{\txt{\'et}}}
 \def\areh[#1]{\ar[#1]|{\txt{$H$-eq}}^{\txt{\'et}}}
 \def\ars[#1]{\ar@{->>}[#1]}
 \newcommand{\dplus}{\ar@{}[d]|{\mbox{$\oplus$}}}
 \newcommand{\dtimes}{\ar@{}[d]|{\mbox{$\times$}}}

\usepackage{amsthm}

\theoremstyle{plain}

\newtheoremstyle{named}{}{}{\itshape}{}{\bfseries}{.}{.5em}{\thmnote{#3 }#1}
\theoremstyle{named}
\newtheorem*{namedtheorem}{Theorem}

\usepackage{amsmath}

\DeclareMathOperator{\ann}{Ann}

\DeclareMathOperator{\coker}{coker}

\DeclareMathOperator{\MM}{M}

\DeclareMathOperator{\soc}{Soc}

\DeclareMathOperator{\supp}{Supp}

\newcommand{\p}{{\mathfrak p}}
\newcommand{\q}{{\mathfrak q}}

\begin{document}
\title{New classes of examples satisfying the three matrix analog of Gerstenhaber's theorem}
\author{Jenna Rajchgot}
\thanks{J.R. is partially supported by NSERC grant RGPIN-2017-05732.}
\author{Matthew Satriano}
\thanks{M.S. is partially supported by NSERC grant RGPIN-2015-05631.}
\date{}


\begin{abstract}
In 1961, Gerstenhaber proved the following theorem: if $k$ is a field and $X$ and $Y$ are commuting $d\times d$ matrices with entries in $k$, then the unital $k$-algebra generated by these matrices has dimension at most $d$. The analog of this statement for four or more commuting matrices is false. The three matrix version remains open. 
We use commutative-algebraic techniques to prove that the three matrix analog of Gerstenhaber's theorem is true for some new classes of examples. 

In particular, we translate this three commuting matrix statement into an equivalent statement about certain maps between modules, and prove that this commutative-algebraic reformulation is true in special cases. We end with ideas for an inductive approach intended to handle the three matrix analog of Gerstenhaber's theorem more generally. 

\end{abstract}

\maketitle

\setcounter{tocdepth}{1}
\tableofcontents


\section{Introduction and statement of results}
\subsection{An overview of Gerstenhaber's theorem and its three matrix analog}
\label{subsec:overview}

Let $k$ be a field and let $\MM_d(k)$ denote the space of $d\times d$ matrices with entries in $k$. 

\begin{question}\label{qn:motivating}
Let $X_1,\dots, X_n\in \MM_d(k)$ be pairwise commuting matrices. Must the unital $k$-algebra that they generate be a finite-dimensional vector space of dimension at most $d$? 
\end{question}

When $n=1$, the answer to Question \ref{qn:motivating} is ``yes'' by the Cayley-Hamilton theorem: $X\in \MM_d(k)$ satisfies its characteristic polynomial, thus $I, X, X^2,\dots, X^{d-1}$ is a vector space spanning set for the algebra generated by $X$. 
 
 When $n\geq 4$, the answer is ``no'' in general. 
The standard $n=4$ counter-example is given by the matrices $E_{13},E_{14},E_{23},E_{24}\in\MM_4(k)$, where $E_{ij}$ denotes the matrix with a $1$ in position $(i,j)$ and $0$s elsewhere. 
These 4 matrices generate a 5-dimensional algebra. 
 
The first interesting case is $n=2$. Here the answer to Question \ref{qn:motivating} is ``yes.'' This result is often called \textbf{Gerstenhaber's theorem} and was proved in \cite{Gerstenhaber}. 
Gerstenhaber's proof was algebro-geometric, and relied on the irreducibility of the commuting scheme $\mathcal{C}(2,d)$ of pairs of $d\times d$ commuting matrices (a fact also proved in the earlier paper \cite{MTT}). 
Some years later, linear algebraic proofs (see \cite{BH, LL}) and commutative algebraic proofs (see \cite{Wadsworth, Bergman}) of Gerstenhaber's theorem were found.  More detailed  summaries on the history and approaches to Gerstenhaber's theorem can be found in \cite{SethurSurvey, HolbrookOmeara}.

The case $n=3$ is still open, and is the subject of this paper. We refer to the following statement as the \textbf{three matrix analog of Gerstenhaber's theorem}.
\begin{statement}
\label{conj:triple-mat}
If $X,Y, Z\in\MM_d(k)$ are matrices which pairwise commute, then the unital $k$-algebra generated by $X$, $Y$, and $Z$ is a finite-dimensional vector space of dimension at most $d$.
\end{statement}
To prove Statement \ref{conj:triple-mat}, one might try to mimic the algebro-geometric proof of Gerstenhaber's theorem. 
This approach succeeds whenever the affine scheme of triples of commuting $d\times d$ matrices, denoted $\mathcal{C}(3,d)$,  
is irreducible. Consequently, Statement \ref{conj:triple-mat} is true when $d\leq 10$ and $k$ is of characteristic $0$ (see \cite{Sivic3} and references therein). However, since $\mathcal{C}(3,d)$ has multiple irreducible components when $d\geq 30$ \cite{Guralnick, HolOmla}, a different approach is necessary to handle the general case. 
%

\subsection{Summary of the main results}\label{sect:results}
In this paper, we use commutative-algebraic methods to study the three matrix analog of Gerstenhaber's theorem. 
We do so by reformulating Statement \ref{conj:triple-mat} in terms of morphisms of modules, our key technical tools being Propositions \ref{prop:conj-reformulation} and \ref{prop:conjecture inductive approach}. 
Although this is nothing more than a simple reformulation, an approach along these lines appears to be new.

We work over an arbitrary field $k$ and let $S = k[x_1,\dots, x_n]$ denote a polynomial ring in $n$ variables. 
One can easily rephrase Question \ref{qn:motivating} on $n$ commuting matrices in terms of $S$-modules. We provide a proof of this in 
Section \ref{sec:module-approach} (see also \cite[\S5]{Bergman}).

\begin{proposition}
\label{prop:S/ann<=dimIntro}
Question \ref{qn:motivating} has answer ``yes''  if and only if for all $S$-modules $N$ which are finite-dimensional over $k$, we have 
\begin{equation}\label{eqn:dimM}
\dim S/\ann(N)\leq\dim N.
\end{equation}
Thus, Statement \ref{conj:triple-mat} is true if and only if inequality \eqref{eqn:dimM} holds for all finite-dimensional $k[x,y,z]$-modules $N$.
\end{proposition}

From the perspective of Proposition \ref{prop:S/ann<=dimIntro}, one can approach Statement \ref{conj:triple-mat} by successively considering modules of increasing complexity. The simplest modules to consider are cyclic ones: when $N=S/I$ it is obvious that equation (\ref{eqn:dimM}) holds. The next simplest case to consider is extensions
\[
0\to S/I\to N\to S/\m\to 0,
\]
where $\m\subseteq S$ is a maximal ideal. This case is much less obvious, 
and is the central focus of this paper. Our main result is that Statement \ref{conj:triple-mat} holds for such modules: 

\begin{theorem}
\label{thm:summaryTheorem}
Let $k$ be an infinite field, $S=k[x,y,z]$, $I\subseteq S$ an ideal of finite colength, and $\m\subseteq S$ a maximal ideal. If $N$ is an extension of $S$-modules
\[
0\to S/I\to N\to S/\m\to 0,
\]
then the inequality (\ref{eqn:dimM}) holds for $N$.
\end{theorem}

Using Theorem \ref{thm:summaryTheorem}, we also obtain the following more general result.

\begin{theorem}
\label{thm:summaryTheorem-more-general}
Let $k$ be an infinite field, $S=k[x,y,z]$, and $N$ an $S$-module which is finite dimension over $k$. Suppose $N=N_1\oplus\dots\oplus N_r$ where for each $i$, one of the following holds:
\begin{enumerate}
\item $N_i$ is a cyclic module, or 
\item $N_i$ is local Gorenstein, or
\item there is an extension
\[
0\to M_i\to N_i\to S/\m_i\to 0
\]
where $\m_i\subseteq S$ is a maximal ideal and 
$M_i=\bigoplus_\ell M_{i,\ell}$ with each $M_{i,\ell}$ a cyclic or local Gorenstein module.
\end{enumerate} 
Then the inequality (\ref{eqn:dimM}) holds for $N$.
\end{theorem}

Before discussing the technique of proof, it is worth remarking why Theorems \ref{thm:summaryTheorem} and \ref{thm:summaryTheorem-more-general} are more difficult than the cyclic case $N=S/I$. One reason is that the cyclic case holds over polynomial rings $S=k[x_1,\dots,x_n]$ in arbitrarily many variables, whereas Theorem \ref{thm:summaryTheorem} is specific to 3 variables. 
Indeed, the technique of proof \emph{must} be specific to 3 variables since there are counter-examples of this form in 4 variables:

\begin{example}[{Theorems \ref{thm:summaryTheorem} and \ref{thm:summaryTheorem-more-general} are false for 4 variables}]
Let $S=k[x,y,z,w]$ and $\m=(x,y,z,w)$. Recall the standard 4 variable counter-example is given by the matrices $E_{13}$, $E_{14}$, $E_{23}$, $E_{24}\in\MM_4(k)$. Via the proof of Proposition \ref{prop:S/ann<=dimIntro}, this corresponds to the $S$-module $N$ given as follows: we have an extension
\[
0\to S/((x,y)+\m^2)\to N\stackrel{\pi}{\to} S/\m\to 0
\]
so that $N$ is generated by the elements of $M:=S/((x,y)+\m^2)$ and an additional element $f\in N$ such that $\pi(f)=1\in S/\m$. The $S$-module structure is given by 
\[
zf=wf=0,\quad xf=z\in M,\quad yf=w\in M.
\]
One checks $\ann(N)=\m^2$ so that
\[
\dim S/\ann(N)=5>4=\dim N
\]
which violates the inequality of Proposition \ref{prop:S/ann<=dimIntro}.
\end{example}

The proof of Theorem \ref{thm:summaryTheorem} consists of two steps. The first (see Theorem \ref{thm:main-special-case}) is showing that if a counter-example of this form exists, then we can reduce to the case where $I=(F_{12},F_{13},F_{23},g)$, the $F_{ij}$ are the maximal minors of a specific $2\times 3$ matrix, $g$ is a non-zero divisor, and $N$ satisfies certain additional properties. The second step (see \S\ref{sec:main-thm}) consists of showing that no such counter-example exists.

The case of Gorenstein modules, and extensions of $S/\m$ by Gorenstein modules, are handled in Section \ref{sec:pf of Gorenstein}. This, together with some preliminary results in Section \ref{sec:module-approach}, quickly reduces Theorem \ref{thm:summaryTheorem-more-general} to the case of Theorem \ref{thm:summaryTheorem}.

As mentioned above, the following is a key tool we use in the proofs of Theorems \ref{thm:summaryTheorem} and \ref{thm:summaryTheorem-more-general}.

\begin{proposition}
\label{prop:conj-reformulation}
Let $S=k[x,y,z]$ be a polynomial ring and $\m=(x,y,z)$. Then Statement \ref{conj:triple-mat} is true if and only if for all finite-dimensional $S$-modules $M$ with $\sqrt{\ann M}=\m$, and all $S$-module maps $\beta\colon\m\to M$, we have
\begin{equation}\label{eqn:betaM}
\dim S/\ann(M)+\dim\beta(\ann M)\leq \dim M+1.
\end{equation}
\end{proposition}


From the perspective of Proposition \ref{prop:conj-reformulation}, Statement \ref{conj:triple-mat} is an assertion about bounding the dimension of the image of $\ann(M)$ under a module map. Along these lines, we obtain an inductive approach to Statement \ref{conj:triple-mat}:

\begin{proposition}
\label{prop:conjecture inductive approach}
Let $S=k[x,y,z]$ and $\m=(x,y,z)$. Then Statement \ref{conj:triple-mat} is equivalent to the following assertion: for all finite-dimensional $S$-modules $M$ with $\sqrt{\ann M}=\m$, all submodules $M'\subseteq M$ with $\dim(M/M')=1$, all ideals $J\subseteq\m$, and all $S$-module maps $\beta\colon J\to M'$, if
\[
\dim J/(J\cap\ann M')+\dim\beta(J\cap\ann M')\leq \dim M'
\]
then
\[
\dim J/(J\cap\ann M)+\dim\beta(J\cap\ann M)\leq \dim M.
\]
\end{proposition}

We prove Proposition \ref{prop:conj-reformulation} in Section \ref{sec:module-approach}, and Proposition \ref{prop:conjecture inductive approach} in Section \ref{sec:inductive approach}.

\subsection*{Notation}
All rings in this paper are commutative with unit. We will let $\supp(M)$ and $\soc(M)$ denote the support, respectively the socle, of a module $M$. Unless otherwise specified, $\dim(M)$ and the word dimension will refer to the dimension of a module $M$ as a vector space over a given base field. 

\subsection*{Acknowledgments}
It is a pleasure to thank Jason Bell, Mel Hochster, and Steven Sam for many enlightening conversations. We thank the Casa Mathem\'atica Oaxaca for the wonderful accommodations where part of this work took place. We are especially grateful to Matt Kennedy for suggesting this problem. We performed many computations in Macaulay2 \cite{M2}, which inspired several ideas.

\section{Reformulating Statement \ref{conj:triple-mat} in terms of module morphisms}
\label{sec:module-approach}
Throughout this section, we fix a field $k$, let $S=k[x_1,\dots,x_n]$ and $\m=(x_1,\dots,x_n)$. Although our primary focus is the three matrix analogue of Gerstenhaber's theorem, our proofs work for arbitrary fields and arbitrarily many matrices. For convenience, consider the following:

\begin{statement}
\label{statement:n-matrix-conj}
The algebra generated by commuting matrices $X_1,X_2,\dots,X_n\in\MM_d(k)$ has dimension at most $d$.
\end{statement}

For $n>3$ this statement is false, $n<3$ it is true, 
and the case when 
$n=3$ is Statement \ref{conj:triple-mat}. Our goal in this section is to prove the following version of Proposition \ref{prop:conj-reformulation} which is valid for $n$ commuting matrices.

\begin{proposition}
\label{prop:conj-reformulation-nvars}
Statement \ref{statement:n-matrix-conj} is true if and only if for all finite-dimensional $S$-modules $M$ with $\sqrt{\ann M}=\m$, and $S$-module maps $\beta\colon\m\to M$, we have
\begin{equation}\label{eqn:Sm extended by M}
\dim S/\ann(M)+\dim\beta(\ann M)\leq \dim M+1.
\end{equation}
\end{proposition}

We begin with a different module-theoretic reformulation of Statement \ref{statement:n-matrix-conj}. This appeared as Proposition \ref{prop:S/ann<=dimIntro} in the introduction.

\begin{proposition}
\label{prop:S/ann<=dim}
Statement \ref{statement:n-matrix-conj} holds if and only if for all $S$-modules $N$ which are finite-dimensional over $k$, we have
\[
\dim S/\ann(N)\leq\dim N.
\]
\end{proposition}
\begin{proof}
Given commuting matrices $X_1,\dots,X_n\in\MM_d(k)$, we obtain an $S$-module structure on $k^{\oplus d}$ where we define multiplication by $x_i$ to be the action of $X_i$. Conversely, given any $S$-module $N$ of dimension $d$ over $k$, after fixing a basis we have $N\simeq k^{\oplus d}$ as $k$-vector spaces; multiplication by $x_1,\dots,x_n$ on $N$ can then be viewed as matrices $X_1,\dots,X_n\in\MM_d(k)$ which commute since $x_ix_j=x_jx_i$.

Let $A$ denote the unital $k$-algebra generated by our commuting matrices. We then have a surjection $\pi\colon S\to A$ given by $\pi(x_i)=X_i$, and so $A\simeq S/\ker\pi$ as $S$-modules. Under the correspondence in the previous paragraph, we have $\ker\pi=\ann(N)$ and so $A=S/\ann(N)$. Therefore, the inequality $\dim A\leq d$ is equivalent to the inequality $\dim S/\ann(N)\leq \dim N$.
\end{proof}

\begin{example}
Following the above proof, we describe some triples of commuting matrices obtained from our main theorem, Theorem \ref{thm:summaryTheorem}.
Let $S = k[x,y,z]$ and $\m = (x,y,z)$. 
Let $I\subseteq S$ be an ideal of finite colength. Each module $N$ that fits into a short exact sequence
\[
0\to S/I\xrightarrow{i} N\xrightarrow{\pi} S/\m\to 0
\]
has an ordered basis of the form $i(m_1),\dots, i(m_{d-1}), f$ where $m_1,\dots, m_{d-1}$ is a basis of $S/I$ and $\pi(f) = 1$ in $S/\m$.  Let $n_j = i(m_{j})$.
Let $X, Y, Z\in \text{M}_{d}(k)$ be determined by multiplication of $x, y, z$ on $N$ and the given basis. To describe the first $d-1$ columns of $X$, observe that each $xn_j$ is a linear combination of $n_1,\dots, n_{d-1}$, and this linear combination is determined by multiplication by $x$ in $S/I$. Thus, if $S/I$ with basis $m_1,\dots, m_{d-1}$ is given, the first $d-1$ columns of $X$ are fixed and do not depend on the choice of extension of $S/I$ by $S/\m$. 
On the other hand, the last column of $X$ may be chosen almost arbitrarily. Indeed, $xf$ is a linear combination of $n_1,\dots, n_{d-1}$ (since $\pi(xf) = 0$ in $S/\m$), but there are no restrictions on which linear combinations may occur: given any linear combination of $n_1,\dots, n_{d-1}$, one can construct an extension $N$ such that $xf$ is the given linear combination. The matrices $Y$ and $Z$ are described analogously. In other words, the matrices $X$, $Y$, and $Z$ take the form
\[
X = \left(
\begin{array}{c|c}
   \raisebox{-15pt}{{\LARGE\mbox{{$X'$}}}}&a_1 \\[-3.5ex]
   &\vdots \\[-0.5ex]
  &a_{d-1}\\
  \hline
    0 \cdots 0 &0\\ 
\end{array}
\right),
\qquad
Y = \left(
\begin{array}{c|c}
   \raisebox{-15pt}{{\LARGE\mbox{{$Y'$}}}}&b_1 \\[-4ex]
  &\vdots \\[-0.5ex]
  &b_{d-1}\\
  \hline
    0 \cdots 0 &0\\ 
\end{array}
\right),
\qquad
Z = \left(
\begin{array}{c|c}
   \raisebox{-15pt}{{\LARGE\mbox{{$Z'$}}}}&c_1 \\[-4ex]
  &\vdots \\[-0.5ex]
  &c_{d-1}\\
  \hline
    0 \cdots 0 &0\\ 
\end{array}
\right),
\]
where $X', Y', Z'\in \text{M}_{d-1}(k)$ are determined by multiplication, in $S/I$, by $x$, $y$, and $z$ respectively, and the entries $a_j$, $b_j$, $c_j$ can be chosen arbitrarily by choosing an appropriate extension of $S/I$ by $S/\m$.
\end{example}

The next observation allows us to handle the case of direct sums of modules and in particular, reduce to the local case. 

\begin{lemma}
\label{l:S/ann<=dim-direct-sums}
Let $N=N_1\oplus\dots\oplus N_r$ where each $N_i$ is an $S$-module which is finite-dimensional over $k$. If $\dim S/\ann(N_i)\leq\dim N_i$ for all $i$, then $\dim S/\ann(N)\leq\dim N$.
\end{lemma}
\begin{proof}
We have $\ann(N)=\bigcap_i \ann(N_i)$, so the diagonal map
\[
S/\ann(N)\to \bigoplus_i S/\ann(N_i)
\]
is injective. As a result,
\[
\dim S/\ann(N)\leq \sum_i\dim S/\ann(N_i)\leq\sum_i \dim N_i=\dim N.\qedhere
\]
\end{proof}

\begin{corollary}
\label{cor:S/ann<=dim-local-case}
If $N$ is an $S$-module that is finite-dimensional over $k$, and $\dim S/\ann(N')\leq\dim N'$ for all localizations $N'$ of $N$, then $\dim S/\ann(N)\leq\dim N$. In particular, Statement \ref{statement:n-matrix-conj} is true if and only if $\dim S/\ann(N)\leq\dim N$ for all $S$-modules $N$ which are finite-dimensional over $k$ and such that $\supp N$ is a point.
\end{corollary}
\begin{proof}
Since $N$ is finite-dimensional as a $k$-vector space, its support consists of finitely many points, and $N$ can be written as a direct sum $N_1\oplus\dots\oplus N_r$ where $N_i$ are the localizations of $N$. By Lemma \ref{l:S/ann<=dim-direct-sums}, the desired inequality for $N$ is implied by that for each of the $N_i$. In particular, $\dim S/\ann(N)\leq \dim N$ for all $N$ if and only if it holds for those $N$ which are supported at a point. The corollary then follows from Proposition \ref{prop:S/ann<=dim}.
\end{proof}

\begin{remark}\label{rem:nilpotent}
Corollary \ref{cor:S/ann<=dim-local-case} is equivalent to the following statement about commuting matrices: ``Statement \ref{statement:n-matrix-conj} holds for all choices of pairwise commuting matrices $X_1, X_2,\dots, X_n\in \MM_d(k)$ if and only if it holds for all choices of $n$ pairwise commuting \emph{nilpotent} matrices.'' To see this, we follow the proof of  Proposition \ref{prop:S/ann<=dim}, using the same notation: if the matrices $X_i$ are nilpotent then one easily checks that $S/\ann N$ is a local ring with maximal ideal $\m$. Conversely, if $S/\ann N$ is local with maximal ideal $\m$ then $\ann N \supseteq \m^c$ for some large enough $c$. Identifying the variable $x_i\in S$ with the matrix $X_i$ as in the proof of Proposition \ref{prop:S/ann<=dim}, and recalling that $S/\ann N\cong A$, we see that each $X_i$ satisfies $X_i^c=0$, hence is nilpotent.
\end{remark}

We next prove Proposition \ref{prop:conj-reformulation-nvars}, and hence Proposition \ref{prop:conj-reformulation}.

\begin{proof}[Proof of Proposition \ref{prop:conj-reformulation-nvars}]
By Corollary \ref{cor:S/ann<=dim-local-case}, we need only show that the inequality in the statement of the theorem is equivalent to $\dim S/\ann(N)\leq\dim N$ for all $S$-modules $N$ which are finite-dimensional over $k$ and such that $\supp N$ is a point. Let $N$ be such a module. By translation, we can assume without loss of generality that its support is the origin, i.e.~$\sqrt{\ann N}=\m$. Then the Jordan-H\"older filtration yields  a short exact sequence
\[
0\to M\to N\to S/\m\to 0
\]
and this corresponds to a class $\alpha\in\ext^1(S/\m,M)$. 

From the short exact sequence
\[
0\to\m\to S\to S/\m\to0
\]
we have a long exact sequence
\[
0\to \hom(S/\m,M)\to M\to \hom(\m,M)\to\ext^1(S/\m,M)\to 0.
\]
So, $\ext^1(S/\m,M)\simeq\hom(\m,M)/M$. Thus, our class $\alpha$ lifts to an $S$-module map $\beta\colon\m\to M$. It is easy to check that $\ann(N)=\ann(M)\cap\ker\beta$, so
\[
\begin{split}
\dim S/\ann(N) & =\dim S/\ann(M)+\dim\ann(M)/(\ann(M)\cap\ker\beta) \\ 
& =\dim S/\ann(M)+\dim\beta(\ann M).
\end{split}
\]
Since $\dim N=\dim M + 1$, the inequality
\[
\dim S/\ann(M)+\dim\beta(\ann M)\leq \dim M+1
\]
is equivalent to $\dim S/\ann(N)\leq\dim N$.
\end{proof}

\begin{remark}
\label{rmk:equivalence of 2 inequalities for m mapsto M}
The proof of Proposition \ref{prop:conj-reformulation-nvars} shows that if
\[
0\to M\to N\to S/\m\to 0
\]
is the extension corresponding to the map $\beta\colon \m\to M$, then the inequality (\ref{eqn:Sm extended by M}) holds if and only if the inequality (\ref{eqn:dimM}) holds for $N$. Note that $N$ only depends on the class $[\beta]\in\ext^1(S/\m,M)$. Consequently, whether or not the inequality (\ref{eqn:Sm extended by M}) holds depends only on the class $[\beta]$ rather than the map $\beta$ itself.
\end{remark}

In light of Proposition \ref{prop:conj-reformulation-nvars}, we make the following definition:

\begin{definition}
\label{def:beta-M-counter-ex}
Let $M$ be an $S$-module which is finite-dimensional as a $k$-vector space with $\sqrt{\ann M}=\m$. If $\beta\colon\m\to M$ is an $S$-module map, we say that $(\beta,M)$ is a \emph{counter-example} if $\dim S/\ann(M)+\dim\beta(\ann M)> \dim M+1$.
\end{definition}


\begin{corollary}
\label{cor:conj-reformulation-nvars-cyclic}
If $I\subseteq S$ is an ideal with $S/I$ finite-dimensional over $k$, then $(\beta,S/I)$ is a counter-example if and only if $\dim\beta(I)\geq2$.
\end{corollary}
\begin{proof}
Let $M=S/I$ and notice that $\ann(S/I)=I$. By definition, $(\beta,S/I)$ is a counter-example if and only if 
\[
\dim S/I+\dim\beta(I)=\dim S/\ann(M)+\dim\beta(\ann M)> \dim M+1=\dim S/I+1.
\]
In other words, it is a counter-example if and only if $\dim\beta(I)>1$.
\end{proof}

We end with the following result which reduces our search for counter-examples $(\beta,M)$ to the case where $M$ is indecomposable. Note the distinction from Lemma \ref{l:S/ann<=dim-direct-sums}: the lemma concerns the inequality $\dim S/\ann(M)\leq\dim M$ whereas the proposition below concerns the inequality $\dim S/\ann(\widetilde{M})\leq\dim \widetilde{M}$, where $\widetilde{M}$ is the extension of $S/\m$ by $M$ defined by $\beta$.

\begin{proposition}
\label{prop:direct-sums}
Let $\beta\colon\m\to M$ be an $S$-module map with $M$ finite-dimensional over $k$ and satisfying $\sqrt{\ann M}=\m$. Suppose $M=N_1\oplus N_2$ and let $\pi_j\colon M\to N_j$ be the two projections. If neither $(\pi_1\beta,N_1)$ nor $(\pi_2\beta,N_2)$ is a counter-example, then $(\beta,M)$ is also not a counter-example.
\end{proposition}
\begin{proof}
Let $\widetilde{M}$ be the extension defined by $\beta$, and let $\widetilde{M}_j$ be the extension defined by $\pi_j\beta$. We must show $\dim S/\ann(\widetilde{M})\leq\dim\widetilde{M}$. We know that
\[
\ann(\widetilde{M})=\ker\beta\cap\ann(M)=\bigcap_j\ker(\pi_j\beta)\cap\bigcap_j\ann(N_j)=\bigcap_j\ann(\widetilde{M}_j),
\]
so we have
\[
\dim S/\ann(\widetilde{M}) =\sum_j\dim S/\ann(\widetilde{M}_j)-\dim S/(\ann(\widetilde{M}_1)+\ann(\widetilde{M}_2)).
\]
Since $(\pi_j\beta,N_j)$ is not a counter-example, we have $\dim S/\ann(\widetilde{M}_j)\leq \dim\widetilde{M}_j$. Also notice that $\sqrt{\ann{M}}=\m$ implies $\ann(\widetilde{M}_1)+\ann(\widetilde{M}_2)\subseteq\m$. Hence, $\dim S/(\ann(\widetilde{M}_1)+\ann(\widetilde{M}_2))\geq1$. Since
\[
\dim\widetilde{M}=1+\sum_j\dim N_j=-1+\sum_j\dim\widetilde{M}_j,
\]
we have our desired inequality $\dim S/\ann(\widetilde{M})\leq \dim\widetilde{M}$.
\end{proof}

\section{Reducing Theorem \ref{thm:summaryTheorem} to a special case}
\label{sec:reducing main theorem to special case}

Throughout this section, we fix an infinite field $k$, let $S=k[x_1,x_2,x_3]$ and $\m=(x_1,x_2,x_3)$. The goal of the next two sections is to prove Theorem \ref{thm:summaryTheorem}. By Lemma \ref{l:S/ann<=dim-direct-sums}, we reduce immediately to the case where $\supp(S/I)=\m$, so we must prove:

\begin{theorem}
\label{thm:main}
If $M$ is cyclic and $\supp M = \m$, then $(\beta,M)$ is not a counter-example in the sense of Definition \ref{def:beta-M-counter-ex}.
\end{theorem}

%
%
The focus of this section is to prove the following theorem, which reduces Theorem \ref{thm:main} to a special case:


\begin{theorem}
\label{thm:main-special-case}
Suppose $(\beta,M)$ is a counter-example 
over $k$ with $M$ cyclic and $\sqrt{\ann M}=\m$. Then there exist $f_1,f_2,f_3\in\m$ and $g,h\in\ann(M)$ with the following properties:
\begin{enumerate}
\item\label{item::dimSJ} letting $F_{ij}=x_if_j-x_jf_i$ and $J=(F_{12},F_{13},F_{23},g)$, we have $h\notin J$, the module $S/J$ is finite-dimensional over $k$, and $\m\in\supp(S/J)$,
\item\label{item::lift-beta-SJ} letting $\beta'\colon\m\to S/J$ be the $S$-module map defined by $\beta'(x_i)=f_i$, the elements $\beta'(g)$ and $\beta'(h)$ are linearly independent in the localization of $S/(J+\<h\>)$ at $\m$, and
\item\label{item::socSJ2} $\dim\soc(S_\m/J_\m)=2$.
\end{enumerate}
\end{theorem}

This theorem is proved over the course of \S\S\ref{subsec:nzd}--\ref{subsec:socle}. We begin with some preliminary results. Since $M$ is cyclic, it is of the form $S/I$ where $I$ is an ideal of $S$ with $\sqrt{I}=\m$. By Corollary \ref{cor:conj-reformulation-nvars-cyclic}, we know $\dim\beta(I)\geq2$. Furthermore, we can make a minimality assumption: we may assume that $S/I$ is the cyclic module of smallest dimension for which there exists a counter-example $(\beta,S/I)$, i.e.~for all cyclic modules $S/K$ with $\dim S/K<\dim S/I$ and all pairs $(\gamma,S/K)$, we can assume $\dim\gamma(K)\leq1$.

Let $f_i\in S$ such that $\beta(x_i)=f_i$ mod $I$. Letting $F_{ij}=x_if_j-x_jf_i$, we see $F_{ij}\in I$. So, we can write
\[
I = (F_{12}, F_{13}, F_{23}, g_1,\dots, g_s)
\]
for some polynomials $g_1,\dots, g_s\in \frak{m}$.  

\begin{lemma}\label{lem:genset}
$\beta(I)\subseteq\soc(S/I)$ is the vector space spanned by the $\beta(g_i)$. Moreover, there exist $i\neq j$ such that $\beta(g_i)$ and $\beta(g_j)$ are linearly independent.
\end{lemma}
\begin{proof}
Notice that if $p\in I$, then $x_i\beta(p)=p\beta(x_i)=0$ in $S/I$. This shows that $\beta(I)\subseteq\soc(S/I)$.
Next observe that
\[
\beta(F_{ij})=\beta(x_i)f_j-\beta(x_j)f_i=f_if_j-f_jf_i=0.
\]
As a result, $\beta(I)$ is the ideal generated by the $\beta(g_j)$ and since the $\beta(g_j)$ are contained in $\soc(S/I)$, this ideal is nothing more than the vector space they span. Since $\dim\beta(I)\geq2$, there must be $i\neq j$ such that $\beta(g_i)$ and $\beta(g_j)$ are linearly independent.
\end{proof}

\begin{proposition}\label{prop:furtherReduction}
For each $1\leq i\leq 3$, we have $f_i\in \frak{m}$.
\end{proposition}

\begin{proof}
Without loss of generality, assume that $f_1\notin \frak{m}$. Since $\sqrt{I}=\m$, we see $S/I$ is an Artin local ring, and so $f_1$ is a unit in $S/I$. Since $x_2f_1-x_1f_2=F_{12}$ is $0$ in $S/I$, we see $x_2 = x_1f_2f_1^{-1}\in(x_1)$. Similarly for $x_3$.

Thus, the maximal ideal $\m/I$ of $S/I$ is principally generated by $x_1$. Any Artin local ring with principal maximal ideal has the property that all ideals are of the form $(\m/I)^n$, hence principal. Since $\soc(S/I)$ is an ideal, it is principal and so 1-dimensional. This contradicts Lemma \ref{lem:genset} which shows that $\dim\soc(S/I)\geq2$.
\end{proof}

\subsection{Showing the existence of a non-zero divisor}
\label{subsec:nzd}
In this subsection, we show the existence of $g$ and $h$ in the statement of Theorem \ref{thm:main-special-case}. Let us give an outline of how we proceed. We begin by showing that $S/(F_{12},F_{13},F_{23})$ is Cohen-Macaulay of Krull dimension 1. Now there are of course many choices of $g$ such that $S/(F_{12},F_{13},F_{23},g)$ has Krull dimension 0 (i.e.~finite dimensional over $k$), however to prove Theorem \ref{thm:main-special-case}, we need to guarantee that $\dim\soc(S_\m/J_\m)=2$ and that $\beta'(g)$ and $\beta'(h)$ are linearly independent. This is accomplished by choosing $g$ to be a suitable non-zero divisor in $S/(F_{12},F_{13},F_{23})$.

\begin{proposition}\label{prop:dim1}
Every minimal prime of $S$ over $(F_{12}, F_{13}, F_{23})$ has height $2$.
\end{proposition}

\begin{proof}
For convenience, let $L = (F_{12}, F_{13}, F_{23})$. 
As $F_{12}, F_{13}, F_{23}$ are the minors of the $2\times 3$ matrix
\[ 
\begin{pmatrix}x_1&x_2&x_3\\f_1&f_2&f_3 \end{pmatrix},
\]
the height of each minimal prime over $L$ is at most $2$. 
So assume that there is a minimal prime $\p$ over $L$ of height $\leq 1$.

If $\p$ has height $0$, then $\p = \{0\}$ and so $L=\{0\}$. Then $F_{ij} = 0$, so $x_if_j = x_jf_i$ for each $i,j$. From this it follows that there exists $q\in S$ such that $f_i=x_iq$ for all $i$. Thus, given any $h\in \m$, we have $\beta(h) = hq$. In particular, for $h\in I$ we see $\beta(h)=hq\in I$, so $\dim\beta(I) = 0$, which contradicts $\dim\beta(I)\geq2$. 

If $\p$ has height $1$, then there exists $p\in S$ irreducible with $ \p = (p)$. 
By Bertini's theorem, we know that for a generic linear combination $y=\sum_i\lambda_ix_i$, the ideal $(p,y)$ is prime. 
Choose $y$ so that $(p,y)$ is prime and so that (the open conditions) $\lambda_3\neq 0$ and $y\notin (p)$ are satisfied.
Let $f_y = \sum_i\lambda_i f_i$, and let $F_{1y} = x_1f_y-yf_1, F_{2y} = x_2f_y-yf_2$. 
Observe that $L =  ( F_{12}, F_{1y}, F_{2y})$. 



Next, since $F_{1y},F_{2y}\in L\subseteq (p)\subsetneq(p, y)$, we have $x_1f_y,x_2f_y\in (p, y)$. Recalling that $(p,y)$ is prime, we see $f_y\in (p, y)$, or both $x_1, x_2 \in (p, y)$. 
In the latter case, we have that $(x_1,x_2,x_3)=(x_1,x_2,y) \subseteq (p,y)$, which is impossible as the vanishing locus $V(p,y)\subseteq \mathbb{A}^3$ is irreducible of dimension at least $1$. 
We must therefore have $f_y\in (p, y)$.

Since $f_y\in(p,y)$, we have $f_y = qy + r_yp$ for some $q, r_y\in S$. Using that $p$ divides $F_{1y} = x_1f_y - yf_1=(x_1q-f_1)y+x_1r_yp$, we see that $p$ divides $(x_1q-f_1)y$. Since $(p)$ is prime and $y\notin p$, we have $x_1q-f_1\in(p)$ and so $f_1 = qx_1 + r_1p$ for some $r_1$. Similarly, $f_2 = qx_2 + r_2p$ for some $r_2$. As a result, $\beta=\beta'+\beta''$ where $\beta'(h)=qh$ for all $h\in\m$, and $\beta''(x_i)=pr_i$. By Remark \ref{rmk:equivalence of 2 inequalities for m mapsto M}, whether or not $(\beta,S/I)$ is a counter-example depends only on the value of $[\beta]\in\ext^1(S/\m,S/I)$ and since $[\beta']=0$, we can assume $\beta=\beta''$. As a result, we can assume the image of $\beta$ factors through $(p)\subsetneq S/I$.
Since $(p)$ is generated by a single element, we have $(p) \simeq S/J$ where $J = \ann(p)$. Since $\dim S/J<\dim S/I$, by our minimality assumption at the start of \S\ref{sec:reducing main theorem to special case}, we know that $\beta\colon\m\to (p)=S/J$ is not a counter-example, and so $\dim \beta(J) \leq 1$. But, $I \subseteq J$ because $I$ kills $p$. So, $\dim \beta(I) \leq \dim \beta(J) \leq 1$.
\end{proof}

\begin{corollary}
\label{cor:CM}
$S/(F_{12}, F_{13}, F_{23})$ is Cohen-Macaulay of Krull dimension $1$.
\end{corollary}

\begin{proof}
By Proposition \ref{prop:dim1}, the variety cut out by $(F_{12}, F_{13}, F_{23})$ has codimension 2 in $\AA^3$, so $S/(F_{12}, F_{13}, F_{23})$ has Krull dimension $1$. Since $F_{12}, F_{13}, F_{23}$ are the $2\times 2$ minors of a $2\times 3$ matrix, we may apply \cite[Theorem 18.18]{Ei} (originally proven in \cite{HochsterEagon}) to see that $S/(F_{12}, F_{13}, F_{23})$ is Cohen-Macaulay.
\end{proof}

The following proposition establishes the existence of our desired $g,h\in\m$.

\begin{proposition}\label{prop:counterForm}
There exist $g,h\in I$ such that $g$ is a non-zero divisor in $S/(F_{12}, F_{13}, F_{23})$ and $\beta(g)$ and $\beta(h)$ are linearly independent. Furthermore, we necessarily have $h\notin(F_{12}, F_{13}, F_{23},g)$.
\end{proposition}

\begin{proof}



Recall our notation $I=(F_{12},F_{13},F_{23},g_1,\dots,g_s)$. 
Our first goal is to show that $I$ contains a non-zero divisor in $S/(F_{12}, F_{13}, F_{23})$. 

We know from Corollary \ref{cor:CM} that $S/(F_{12}, F_{13}, F_{23})$ is Cohen-Macaulay, so its associated primes are its minimal primes. Consequently, the set of zero divisors of $S/(F_{12}, F_{13}, F_{23})$ is the union of minimal primes over $(F_{12}, F_{13}, F_{23})$. 
If $I$ is contained in this union of minimal primes then, by the prime avoidance lemma, $I$ is contained in one of these minimal primes.
But this is impossible as $S/(F_{12}, F_{13}, F_{23})$ is Cohen-Macaulay of Krull dimension 1 by Corollary \ref{cor:CM}, and $S/I$ has Krull dimension $0$ by assumption.
Thus, $I$ is not contained in the union of minimal primes over $(F_{12}, F_{13}, F_{23})$ and so $I$ contains a non-zero divisor of $S/(F_{12}, F_{13}, F_{23})$.

Next, by Lemma \ref{lem:genset}, there exist $i\neq j$ such that $\beta(g_i)$ and $\beta(g_j)$ are linearly independent in $S/I$. 
Let $q\in I$ be a non-zero divisor of $S/(F_{12}, F_{13}, F_{23})$. If $c,d\in k\setminus\{0\}$ with $c\neq d$, and if $q+cg_i$ and $q+dg_i$ are in the same minimal prime over $(F_{12}, F_{13}, F_{23})$, then $q$ and $g_i$ are also in that minimal prime, a contradiction to $q$ being a non-zero divisor. Thus, since there are only finitely many minimal primes, 
we see that for all but finitely many $c\in k$, the polynomial $q+cg_i$ is a non-zero divisor. 

Since $q\in I$, we know from Lemma \ref{lem:genset} that $\beta(q)\in\soc(S/I)$. Then $\beta(g_i)$ and $\beta(q)$ are elements of the vector space $\soc(S/I)$ which has dimension at least 2 and $\beta(g_i)\neq0$, so for infinitely many $c\in k$, we see $\beta(q+cg_i)=\beta(q)+c\beta(g_i)\neq 0$. Combining this with our conclusion from the previous paragraph that $q+cg_i$ is a non-zero divisor for all but finitely many $c\in k$, we see we can find a non-zero divisor $g:=q+cg_i\in I$ such that $\beta(g)\neq 0$. Now choose $h\in \m$ to be any linear combination of $g_i$ and $g_j$ such that $\beta(h)$ is not a scalar multiple of $\beta(g)$; this is possible by Lemma \ref{lem:genset} as $\beta(g_i)$ and $\beta(g_j)$ span a 2-dimensional subspace of $\soc(S/I)$.

Lastly, we show $h\notin J:=(F_{12}, F_{13}, F_{23},g)$. From Lemma \ref{lem:genset}, we know $0\neq\beta(g)\in\soc(S/I)$ and we see that $\beta(F_{ij})=0$. Thus, $\beta(J)$ is 1-dimensional, generated by $\beta(g)$. Since $\beta(g)$ and $\beta(h)$ are linearly independent, it follows that $h\notin J$.
\end{proof}

\begin{remark}[Hypothesis that $k$ is infinite]
Proposition \ref{prop:counterForm} is the only step in the proof of Theorem \ref{thm:main} that assumes that $k$ is infinite.
\end{remark}

With the choice of $g$ and $h$ from Proposition \ref{prop:counterForm}, we have:
\begin{corollary}\label{cor:conditions12}
Conditions (\ref{item::dimSJ}) and (\ref{item::lift-beta-SJ}) of Theorem \ref{thm:main-special-case} hold.
\end{corollary}
\begin{proof}
From Proposition \ref{prop:counterForm}, $g$ is a non-zero divisor of $S/(F_{12},F_{13},F_{23})$, which is Cohen-Macaulay of Krull dimension 1 by Corollary \ref{cor:CM}. Thus, $S/J=S/(F_{12},F_{13},F_{23},g)$ has Krull dimension 0, so it is finite-dimensional over $k$. Since $S/J$ surjects onto $S/I$ and $\sqrt{I}=\m$, we know that $\m\in\supp(S/J)$, which proves condition (\ref{item::dimSJ}).

Next since $h\in I$, we have surjections $S/J\to S/(J+\<h\>)\to S/I$. After localizing at $\m$, these remain surjections. Since $\sqrt{I}=\m$, we know $S_\m/I_\m=S/I$ and so 
\[
\m\stackrel{\beta'}{\longrightarrow} S/J\to S_\m/J_\m\to S_\m/(J+\<h\>)_\m
\]
is a lift of $\beta$, meaning that after post-composing the above map by $S_\m/(J+\<h\>)_\m\to S_\m/I_\m=S/I$, we obtain $\beta$. Since $\beta(g)$ and $\beta(h)$ are linearly independent in $S/I$, it must also be the case that $\beta'(g)$ and $\beta'(h)$ are linearly independent in $S_\m/(J+\<h\>)_\m$, proving condition (\ref{item::lift-beta-SJ}).
\end{proof}

\subsection{Computing the socle}
\label{subsec:socle}
In this subsection we compute the socle of $S_\m/J_\m$, thereby showing condition (\ref{item::socSJ2}) of Theorem \ref{thm:main-special-case} and finishing the proof.

\begin{proposition}
\label{prop:dimSocSJ}
$\dim\soc(S_\m/J_\m)=2$.
\end{proposition}

\begin{proof}
From Corollary \ref{cor:CM} we know that $S/(F_{12}, F_{13}, F_{23})$ is Cohen-Macaulay of dimension 1. Thus, the Eagon-Northcott complex yields a minimal free resolution \cite[Theorem A2.60]{Ei2}:
\[
 0\rightarrow S^2\xrightarrow{A} S^3 \xrightarrow{B} S \rightarrow S/\langle F_{12}, F_{13}, F_{23}\rangle \rightarrow 0,~~~A = \begin{pmatrix} x_3&f_3\\ x_2&f_2\\ x_1&f_1 \end{pmatrix}, B = \begin{pmatrix} F_{12}& -F_{13}& F_{23}\end{pmatrix}
\]
is an exact sequence.

Since $g\in \frak{m}$ is a nonzerodivisior in $S/(F_{12}, F_{13}, F_{23})$, we can use the above resolution to obtain the minimal free resolution of $S/J$:
\[
 0\rightarrow S^2\xrightarrow{A'} S^5 \xrightarrow{B'} S^4 \xrightarrow{C'} S \rightarrow S/J \rightarrow 0,\]
 \[~~~A' = \begin{pmatrix} g&0\\0&g\\ x_3&f_3\\ x_2&f_2\\ x_1&f_1 \end{pmatrix},~~ B' = \begin{pmatrix}x_3&f_3&-g&0&0\\x_2&f_2&0&-g&0\\x_1&f_1&0&0&-g\\ 0&0&F_{12}&-F_{13}&F_{23} \end{pmatrix}, ~~
 C' = \begin{pmatrix} F_{12}& -F_{13}& F_{23}&g\end{pmatrix}
\]
Localizing at $\frak{m}$, we obtain the minimal free resolution
\[
0\rightarrow S_{\frak{m}}^2\rightarrow S_{\frak{m}}^5\rightarrow S_{\frak{m}}^4\rightarrow S_{\frak{m}}\rightarrow (S/J)_{\frak{m}}\rightarrow 0.
\]
Consequently, $\dim\soc(S_\m/J_\m)=2$, as the left-most term in the above resolution is of rank $2$.
\end{proof}



\section{Completing the proof of Theorem \ref{thm:summaryTheorem}}
\label{sec:main-thm}
Having now proved Theorem \ref{thm:main-special-case}, we have reduced Theorem \ref{thm:main} 
to showing the following. Let $f_1,f_2,f_3,g\in\m$, $F_{ij}=x_if_j-x_jf_i$, and $J=(F_{12},F_{13},F_{23},g)$ such that $S/J$ is finite-dimensional with $\dim\soc(S_\m/J_\m)=2$. Let $\beta\colon\m\to S/J$ be defined by $\beta(x_i)=f_i$. Then it is impossible to find $h\in\m\setminus J$ such that $\beta(g)$ and $\beta(h)$ are linearly independent in $S_\m/(J_\m+\<h\>)$.

We begin with two well-known lemmas.
\begin{lemma}
\label{l:Gor-ann-int}
If $R$ is an Artinian Gorenstein local ring, and $I_1$ and $I_2$ are ideals of $R$, then $\ann(I_1) + \ann(I_2) = \ann(I_1 \cap I_2)$.
\end{lemma}
\begin{proof}
In any commutative ring, we have the equality $\ann(K_1 + K_2) = \ann(K_1) \cap \ann(K_2)$ for all ideals $K_1$ and $K_2$. For Artinian Gorenstein local rings, we have $\ann(\ann(K))=K$ for all ideals $K$, see \emph{e.g.}~\cite[Exercise 3.2.15]{CM-rings}. So letting $K_j=\ann(I_j)$, we see in our case that $\ann( \ann(I_1)  + \ann(I_2) ) = I_1 \cap I_2$. Taking annihilators of both sides then proves the result.
\end{proof}

\begin{lemma}
\label{l:quot--->relation-in-soc}
Let $R$ be any Artinian local ring and $0\neq r\in R$. If $s_1,\dots,s_m$ form a basis for $\soc(R)$, then there is a linear dependence relation among the $s_i$ in $R/r$.
\end{lemma}
\begin{proof}
Since every non-zero ideal of $R$ intersects the socle non-trivially, $(r)\cap\soc(R)$ contains a non-trivial element $s$. We can write $s=\sum a_i s_i$ with $(a_1,\dots,a_m)\in k^m\setminus 0$. Then in $R/r$ we have the linear dependence relation $\sum a_i s_i=0$.
\end{proof}

\begin{proposition}
\label{prop:Gor-divisibility-property}
Let $K\subseteq S=k[x_1,x_2,x_3]$ be an ideal with $S/K$ an Artinian Gorenstein local ring, and $\gamma\colon\m\to S/K$ an $S$-module map. If $\gamma(q)$ is divisible by $q$ for all $q\in\m$, then there exists $r\in S/K$ such that for all $q\in\m$, we have $\gamma(q)=qr$.
\end{proposition}
\begin{proof}
For every $r\in S/K$, let $\delta_r\colon\m\to S/K$ be the map $\delta_r(q)=qr$. To prove the result, it suffices to replace $\gamma$ by $\gamma-\delta_r$ for any $r$. Let $\gamma(x_i)=x_ip_i$. Replacing $\gamma$ by $\gamma-\delta_{p_3}$, we can assume $p_3=0$, i.e.~$\gamma(x_3)=0$. Then $x_1x_3p_1=x_3\gamma(x_1)=x_1\gamma(x_3)=0$. In other words,
\[
p_1\in\ann(x_1x_3).
\]
Similarly,
\[
p_2\in\ann(x_2x_3)\quad\textrm{and}\quad p_1-p_2\in\ann(x_1x_2).
\]

Our first goal is to show that $p_2\in\ann(x_2)+\ann(x_3)$. By Lemma \ref{l:Gor-ann-int}, we have $\ann(x_2)+\ann(x_3)=\ann((x_2)\cap(x_3))$. So, let $q\in(x_2)\cap(x_3)$ and we must show $p_2q=0$. Since $q$ is divisible by both $x_2$ and $x_3$, we can write $q=x_3q'$ and $q=x_2q''$. We can further write $q''=a(x_2)+x_3b(x_2,x_3)+x_1c(x_1,x_2,x_3)$ where $a\in k[x_2]$, $b\in k[x_2,x_3]$, and $c\in k[x_1,x_2,x_3]$. Then using $p_2\in\ann(x_2x_3)$ and $p_1-p_2\in\ann(x_1x_2)$, we have
\[
\begin{split}
p_2q = p_2x_2a(x_2)+p_2x_1x_2c(x_1,x_2,x_3) &= p_2x_2a(x_2)+p_1x_1x_2c(x_1,x_2,x_3)\\
&= \gamma(x_2a(x_2)+x_1x_2c(x_1,x_2,x_3)+x_3d)
\end{split}
\]
for any choice of $d$. Choosing $d=x_2b(x_2,x_3)-q'$, we see
\[
s:=x_2a(x_2)+x_1x_2c(x_1,x_2,x_3)+x_3d=q-x_3q'=0.
\]
But, by assumption $p_2q=\gamma(s)$ is divisible by $s=0$, and hence $p_2q=0$. We have therefore shown
\[
p_2\in\ann(x_2)+\ann(x_3).
\]

Thus, we can write $p_2=(p_2-r)+r$ with $r\in\ann(x_3)$ and $p_2-r\in\ann(x_2)$. Then $(\gamma-\delta_r)(x_2)=x_2(p_2-r)=0$ and $(\gamma-\delta_r)(x_3)=-x_3r=0$, and so we can assume
\[
p_2=p_3=0.
\]
Then
\[
p_1\in\ann(x_1x_2)\cap\ann(x_1x_3).
\]
To finish the proof we need only show that $p_1\in\ann(x_1)+\ann(x_2,x_3)$. Indeed, upon doing so, we can write $p_1=(p_1-r)+r$ with $r\in\ann(x_2,x_3)$ and $p_1-r\in\ann(x_1)$. Then $(\gamma-\delta_r)(x_1)=x_1(p_1-r)=0$, $(\gamma-\delta_r)(x_2)=-x_2r=0$, and $(\gamma-\delta_r)(x_3)=-x_3r=0$. In other words, we will have found $r\in S$ such that $\gamma-\delta_r=0$, i.e.~$\gamma(q)=qr$ for all $q\in\m$.

To show that $p_1\in\ann(x_1)+\ann(x_2,x_3)$, and thereby finish the proof, we again note by Lemma \ref{l:Gor-ann-int} that $\ann(x_1)+\ann(x_2,x_3)=\ann((x_1)\cap(x_2,x_3))$. We let $q\in(x_1)\cap(x_2,x_3)$ and must show that $p_1q=0$. We can then write $q=x_1(a(x_1)+x_2b(x_1,x_2)+x_3c(x_1,x_2,x_3))$ and $q=x_2q'+x_3q''$. Then using that $p_1\in\ann(x_1x_2)\cap\ann(x_1x_3)$, we have
\[
p_1q=p_1x_1a(x_1)=\gamma(x_1a(x_1)+x_2d + x_3e)
\]
for any choice of $d$ and $e$. As before, choosing $d=x_1b(x_1,x_2)-q'$ and $e=x_1c(x_1,x_2,x_3)-q''$ yields $s:=x_1a(x_1)+x_2d + x_3e=0$, and since $p_1q=\gamma(s)$ is divisible by $s=0$, we have $p_1q=0$.
\end{proof}

Finally, we turn to the proof of the main theorem.

\begin{proof}[Proof of Theorem \ref{thm:main}]
Since $S/J$ is finite-dimensional over $k$, we know that $S_\m/J_\m=S/K$ for some ideal $K$. Recall that $\dim\soc(S/K)=2$ and our goal is to show that for all $h\in\m\setminus J$ there is a linear dependence relation between $\beta(g)$ and $\beta(h)$ in $S_\m/(J_\m+\<h\>)=S/(K+\<h\>)$. In particular, we may assume $\beta(g)\neq0$ in $S/(K+\<h\>)$.

To begin, we show $\beta(h)\notin\soc(S/K)$. If $\beta(h)$ were in the socle, then since $\dim\soc(S/K)=2$ and $\beta(g)\in\soc(S/K)$, either we have our desired linear dependence relation between $\beta(g)$ and $\beta(h)$ in $S/K$ (and hence in $S/(K+\<h\>)$), or $\beta(g)$ and $\beta(h)$ form a basis for $\soc(S/K)$. In the latter case, Lemma \ref{l:quot--->relation-in-soc} shows there is a linear dependence relation between $\beta(g)$ and $\beta(h)$ in $S/(K+\<h\>)$. So, we have shown our claim that $\beta(h)\notin\soc(S/K)$. Further note that $h\in\soc(S/K)$ implies $\beta(h)\in\soc(S/K)$, since $x_i\beta(h)=h\beta(x_i)\in h\m=0$. So, we conclude
\[
h,\beta(h)\notin\soc(S/K).
\]

Next, notice that $\beta(h)$ does not divide $\beta(g)$ in $S/K$. Indeed, suppose to the contrary that $\beta(g)=q\beta(h)$ with $q\in S$. If $q\in k$, then we have a linear dependence relation in $S/K$ and hence in $S/(K+\<h\>)$. If $q\in\m$, then $\beta(g)=q\beta(h)=h\beta(q)\in \<h\>$ so $\beta(g)=0$ in $S/(K+\<h\>)$, which is again a linear dependence relation. This shows our claim that
\[
\beta(g)\notin S\beta(h).
\]
Now, $S\beta(h)$ is an ideal of $S/K$, so it intersects $\soc(S/K)$ non-trivially. Since $\dim\soc(S/K)=2$, we know $S\beta(h)\cap\soc(S/K)$ has dimension 1 or 2, but $\beta(g)\notin S\beta(h)\cap\soc(S/K)$, and so $S\beta(h)\cap\soc(S/K)$ is 1-dimensional. Let $q_0\in S$ such that $q_0\beta(h)$ is a basis vector for $S\beta(h)\cap\soc(S/K)$. Then
\[
\beta(g)\textrm{\ \,and\ \,}q_0\beta(h)\textrm{\ \,form\ a\ basis\ for\ }\soc(S/K).
\]

Since $h\notin\soc(S/K)$, we can induct on the smallest $\ell$ for which $\m^\ell h\in\soc(S/K)$. 
That is, we can assume the result for $qh$ for all $q\in\m$, i.e.~we can assume that $\beta(g)$ and $\beta(qh)$ are linearly dependent in $S/(K+\<qh\>)$ for all $q\in\m$. So for all $q\in\m$, there exists $p\in S/K$ and $a,b\in k$ such that $(a,b)\neq (0,0)$ and
\[
a\beta(g)+b\beta(qh)=pqh.
\]
Note that $\beta(qh)=h\beta(q)\in\<h\>$, so the above equality shows $a\beta(g)\in \<h\>$. This yields our desired linear dependence relation among $\beta(g)$ and $\beta(h)$ in $S/(K+\<h\>)$ unless $a=0$, in which case after rescaling we can assume $b=1$. We can therefore assume that
\[
\beta(qh)\in Sqh\ \ \forall q\in\m.
\]

Next, let
\[
\gamma\colon\m\to \<h\>\subseteq S/K,\quad \gamma(q)=\beta(qh).
\]
Since $\<h\>\cap\soc(S/K)$ is non-trivial, it has dimension 1 or 2. If $\beta(g)$ is in this intersection, then we have our desired linear dependence relation among $\beta(g)$ and $\beta(h)$ in $S/(K+\<h\>)$, so we can assume this is not the case. Thus, $\<h\>\cap\soc(S/K)$ does not contain $\beta(g)$, so it is 1-dimensional, and hence $\<h\>$ is Gorenstein. Notice that $\<h\>\simeq S/\ann_{S/K}(h)$ and via this identification, the condition $\beta(qh)\in \<qh\>$ is equivalent to the condition that $q$ divides $\gamma(q)$. Applying Proposition \ref{prop:Gor-divisibility-property}, there is $r\in S/\ann_{S/K}(h)$ such that $\gamma(q)=qr$. Translating this back into a statement about $\<h\>$ via our identification with $S/\ann(h)$, this says
\[
\exists\, r\in S\textrm{\ \,such\ that\ \,}\beta(qh)=qhr\ \ \forall q\in\m.
\]
As a result, $\beta(h)-hr\in\ann(\m)=\soc(S/K)$, and so there exist $a,b\in k$ such that
\[
\beta(h)-hr=a\beta(g)+bq_0\beta(h).
\]
If $q_0\in\m$, then we see $\beta(h)-a\beta(g)=h(b\beta(q_0)+r)=0$ in $S/(K+\<h\>)$ and gives our linear dependence relation. So, $\q_0\notin\m$, i.e.~$q_0$ is a unit. By construction $\beta(g)$ and $q_0\beta(h)$ form a basis for $\soc(S/K)$, and $q_0$ is a unit, so $\beta(g)$ and $\beta(h)$ form a basis for $\soc(S/K)$. Then by Lemma \ref{l:quot--->relation-in-soc}, they have a linear dependence relation in $S/(K+\<h\>)$.
\end{proof}

\section{Proof of Theorem \ref{thm:summaryTheorem-more-general}}
\label{sec:pf of Gorenstein}

We begin with the following result which holds for arbitrarily many variables:

\begin{proposition}
\label{prop:Gorenstein and extensions}
Let $S = k[x_1,\dots, x_n]$ and let $\m = (x_1,\dots, x_n)$. Suppose that $M$ is a finite-dimensional $S$-module with $\supp M = \m$ and $\dim \soc(M) = 1$. Then the following hold:
\begin{enumerate}
\item\label{item::Gorenstein inequality} $\dim S/\ann M\leq \dim M$
\item\label{item::ext of Gorenstein inequality} for every short exact sequence
\[
0\to M\to N\to S/\m\to 0,
\]
we have $\dim S/\ann N\leq \dim N$.
\end{enumerate}
\end{proposition}

\begin{proof}
We first prove (\ref{item::Gorenstein inequality}) by induction on the dimension of $M$. If $\dim M = 1$ then $M\simeq S/\m$ and the result holds trivially. For higher dimensional $M$, recall that $M$ has a composition series
\[
0\subseteq M_1\subseteq \cdots \subseteq M_{r-1}\subseteq M
\]
where $M_i/M_{i-1}\simeq S/\m$. Since the socle of $M_{r-1}$ is contained in the socle of $M$, we also have that $\dim \soc(M_{r-1}) = 1$, and we may apply induction to see that $\dim S/\ann M_{r-1}\leq \dim M_{r-1}$. Then we have a short exact sequence
\[
0\to M_{r-1}\to M_r\to S/\m\to 0
\]
which corresponds to a map $\beta\colon\m\to M_{r-1}$. By Remark \ref{rmk:equivalence of 2 inequalities for m mapsto M}, we need only show
\[
\dim S/\ann M_{r-1}+\dim\beta(\ann M_{r-1})\leq\dim M_{r-1}+1.
\]
By induction, we know $\dim S/\ann M_{r-1}\leq\dim M_{r-1}$. Furthermore, $\beta(\ann M_{r-1})\subseteq \soc(M_{r-1})$, so has dimension at most 1. This proves the desired inequality.

The proof of (\ref{item::ext of Gorenstein inequality}) is entirely analogous. We know that the short exact sequence defining $N$ corresponds to a map $\beta\colon\m\to M$. By (\ref{item::Gorenstein inequality}), we know $\dim S/\ann(M)\leq\dim M$. Since $\beta(\ann M)\subseteq\soc(M)$, we have $\dim\beta(\ann M)\leq1$. Combining these two statements, inequality (\ref{eqn:Sm extended by M}) holds for $\beta$, and so $\dim S/\ann N\leq \dim N$ by Remark \ref{rmk:equivalence of 2 inequalities for m mapsto M}.
\end{proof}

We now turn to the proof of the second main result of this paper.

\begin{proof}[Proof of Theorem \ref{thm:summaryTheorem-more-general}]
By Lemma \ref{l:S/ann<=dim-direct-sums}, we reduce immediately to the case where: (i) $N$ is either cyclic, or (ii) $N$ is local Gorenstein, or (iii) there is an extension
\begin{equation}
\label{eqn:M-N-Sm-extension}
0\to M\to N\to S/\m\to 0
\end{equation}
where $\m\subseteq S$ is a maximal ideal and 
$M=\bigoplus_\ell M_{\ell}$ with each $M_{\ell}$ a cyclic or local Gorenstein module. For case (i), inequality (\ref{eqn:dimM}) holds trivially. Case (ii) is handled by Proposition \ref{prop:Gorenstein and extensions} (\ref{item::Gorenstein inequality}).

It remains to handle case (iii). Further decomposing if necessary, we can assume each $M_\ell$ is local. Next, letting $L$ be the set of $\ell$ for which $\supp(M_\ell)=\m$, we can write $N=N'\oplus\bigoplus_{\ell\notin L}M_\ell$ where we have a short exact sequence
\[
0\to \bigoplus_{\ell\in L}M_\ell\to N'\to S/\m\to 0.
\]
Since $M_\ell$ is cyclic or local Gorenstein for every $\ell\notin L$, another application of Lemma \ref{l:S/ann<=dim-direct-sums} combined with cases (i) and (ii) above allows us to assume $N=N'$, i.e.~we can assume that $\supp(M_\ell)=\m$ for all $\ell$. Proposition \ref{prop:direct-sums} then reduces us to the case where there is only one $\ell$; that is, we need only consider extensions (\ref{eqn:M-N-Sm-extension}) where $\sqrt{\ann M}=\m$ and $M$ itself is cyclic or Gorenstein. If $M$ is Gorenstein, then inequality (\ref{eqn:dimM}) holds by Proposition \ref{prop:Gorenstein and extensions} (\ref{item::ext of Gorenstein inequality}). If $M$ is cyclic, then the inequality holds by Theorem \ref{thm:summaryTheorem}.
\end{proof}





\section{Some other instances where Statement \ref{conj:triple-mat} holds}\label{sect:otherInstances}

In this short section we record a couple of additional situations where a finite-dimensional module $M$ over $S=k[x_1,\dots, x_n]$ satisfies $\dim S/\ann M\leq \dim M$.

\begin{proposition}
If $\beta:\m\rightarrow M$ is surjective, then the pair $(\beta, M)$ is not a counterexample. Moreover, if $N$ is the extension of $M$ defined by $\beta$ then 
\[
\dim S/\ann(N) = \dim N.
\]
\end{proposition}

\begin{proof}
Since $\ann(N)=\ann(M)\cap\ker(\beta)$, we see
\[
\begin{split}
\dim S/\ann(N)&=\dim S/\m+\dim \m/\ker(\beta)+\dim\ker(\beta)/(\ann(N)\cap\ker(\beta))\\
&=1+\dim M+\dim\ker(\beta)/(\ann(N)\cap\ker(\beta)).
\end{split}
\]
Since $\dim N=1+\dim M$, we must show $\ker(\beta)\subseteq \ann(M)$. Given $m\in M$ and $f\in \ker(\beta)$, we know $\beta$ is surjective, so $m = \beta(g)$ for some $g\in\m$. Then $fm =f\beta(g) = g\beta(f) =0$, and so $f\in\ann(M)$.
\end{proof}

\begin{proposition}\label{prop:ann(m)=annM}
If there exists $m\in N$ such that $\ann(m) = \ann(N)$, then \[\dim S/\ann N\leq \dim N.\]
\end{proposition}
\begin{proof}
In this case, $\dim S/\ann(N)=\dim S/\ann(m)$ and since $S/\ann(m)$ is isomorphic to the cyclic submodule $Sm\subseteq N$, we necessarily have $\dim S/\ann(m)=\dim Sm\leq \dim N$.
\end{proof}

We end this section with an example where Theorem \ref{thm:summaryTheorem} applies but Proposition \ref{prop:ann(m)=annM} does not.
%


\begin{example}
Let $I = (x,y^2,z)\subseteq k[x,y,z]$. %
Let $N = (xy, z)/(yz, x^2, z^2, xy^2-xz)$ and observe that $N$ fits into a short exact sequence
\[
0\rightarrow S/I \rightarrow N \rightarrow S/\m \rightarrow 0
\]
where the injective map sends $1$ to $xy$. We know by Theorem \ref{thm:summaryTheorem} that inequality (\ref{eqn:dimM}) holds for $N$. Proposition \ref{prop:ann(m)=annM}, however does not apply here: every element of $N$ can be represented as $m=az+bxy+cxz$, for some $a,b,c\in k$, and one checks that there is no choice of $a,b,c\in k$ such that $\ann(m)$ agrees with $\ann N= (z,y^2,xy,x^2)$. Indeed, if $a\neq 0$, then $y-(b/a)x\in \ann(m)$ and if $a = 0$, then $x\in \ann(m)$.
\end{example}

\section{An inductive approach to Statement \ref{conj:triple-mat}}
\label{sec:inductive approach}
Our goal is to prove Proposition \ref{prop:conjecture inductive approach} stated in the introduction. We do so after a preliminary lemma.

\begin{lemma}
\label{l:extension of M by SJ}
Let $S=k[x_1,\dots,x_n]$ and $\m=(x_1,\dots,x_n)$. Then Statement \ref{statement:n-matrix-conj} is true if and only if for all ideals $J\subseteq \m$, all finite-dimensional $S$-modules $M$ with $\sqrt{\ann M}=\m$, and all $S$-module morphisms $\beta\colon J\to M$, we have
\begin{equation}\label{eqn:extension of SJ by M}
\dim J/(J\cap\ann M)+\dim\beta(J\cap\ann M)\leq \dim M.
\end{equation}
\end{lemma}
\begin{proof}
Notice that if $J=\m$, then $J\cap\ann M=\ann M$, and so the inequalities (\ref{eqn:Sm extended by M}) and (\ref{eqn:extension of SJ by M}) are equivalent. So by Proposition \ref{prop:conj-reformulation-nvars}, the inequality (\ref{eqn:extension of SJ by M}) implies Statement \ref{statement:n-matrix-conj}.

Thus, it remains to show that Statement \ref{statement:n-matrix-conj} implies inequality (\ref{eqn:extension of SJ by M}). To see this, fix $J\subseteq\m$ and an $S$-module map $\beta\colon J\to M$. As in the proof of Proposition \ref{prop:conj-reformulation-nvars}, the map $\beta$ defines an extension
\[
0\to M\to N\to S/J\to 0
\]
with $\ann(N)=\ann(M)\cap \ker\beta$. By Proposition \ref{prop:S/ann<=dim}, we know Statement \ref{statement:n-matrix-conj} implies inequality (\ref{eqn:dimM}) for $N$. Now note that
\[
\begin{split}
\dim S/\ann(N) & =\dim S/(J\cap\ann M)+\dim(J\cap\ann M)/(\ker\beta\cap\ann M) \\ 
& =\dim S/(J\cap\ann M)+\dim\beta(J\cap\ann M).
\end{split}
\]
Since $\dim N=\dim M+\dim S/J$, we obtain inequality (\ref{eqn:extension of SJ by M}) by subtracting $\dim S/J$ from both sides of the inequality (\ref{eqn:dimM}).
\end{proof}

\begin{remark}
\label{rmk:equivalence of 2 inequalities}
The proof of Lemma \ref{l:extension of M by SJ} shows that if
\[
0\to M\to N\to S/J\to 0
\]
is the extension corresponding to the map $\beta\colon J\to M$, then the inequality (\ref{eqn:extension of SJ by M}) holds if and only if the inequality (\ref{eqn:dimM}) holds for $N$.
\end{remark}

\begin{proof}[Proof of Proposition \ref{prop:conjecture inductive approach}]
Let $S=k[x,y,z]$ be a polynomial ring and $\m=(x,y,z)$. We know by Lemma \ref{l:extension of M by SJ} that Statement \ref{conj:triple-mat} is true if and only if inequality (\ref{eqn:extension of SJ by M}) holds for all $J$, $M$, and maps $\beta\colon J\to M$. 
So, if Statement \ref{conj:triple-mat} is true, then both of the inequalities in the statement of Proposition \ref{prop:conjecture inductive approach} are true, hence the first inequality implies the second.

We now show that if the implication of inequalities in the statement of Proposition \ref{prop:conjecture inductive approach} holds, then Statement \ref{conj:triple-mat} is true. 
By virtue of Lemma \ref{l:extension of M by SJ}, we need only show that inequality (\ref{eqn:extension of SJ by M}) holds for all $J$, $M$, and $\beta\colon J\to M$. We prove this latter statement by induction on $\dim M$, the base case being trivial. So, we need only handle the induction step. For this, we can choose a submodule $M'\subseteq M$ such that $\dim(M/M')=1$. Then $M/M'\simeq S/\m$ and we let $\pi\colon M\to S/\m$ be the quotient map.

Suppose first that $\pi\beta\colon J\to S/\m$ is surjective. Then letting $I=\ker(\pi\beta)$, we have a map of short exact sequences
\[
\xymatrix{
0 \ar[r] & I\ar[r]\ar[d]^-{\alpha} & J\ar[r]\ar[d]^-{\beta} & S/\m\ar[r]\ar[d]^-{\simeq} & 0\\
0 \ar[r] & M'\ar[r] & M\ar[r]^-{\pi} & S/\m\ar[r] & 0
}
\]
The maps $\alpha$ and $\beta$ define extensions
\[
0\to M'\to N'\to S/I\to 0\quad\textrm{and\ }\quad 0\to M\to N\to S/J\to 0,
\]
respectively, and one checks that $N'\simeq N$. Since $\dim M'<\dim M$, we can assume by induction that 
\[
\dim I/(I\cap\ann M')+\dim\alpha(I\cap\ann M')\leq \dim M'.
\]
By Remark \ref{rmk:equivalence of 2 inequalities}, this is equivalent to the inequality $\dim S/\ann(N')\leq\dim N'$. Since $N\simeq N'$ we obtain the inequality $\dim S/\ann(N)\leq\dim N$, and applying Remark \ref{rmk:equivalence of 2 inequalities} again, we have
\[
\dim J/(J\cap\ann M)+\dim\beta(J\cap\ann M)\leq \dim M.
\]

It remains to handle the case when $\pi\beta\colon J\to S/\m$ is not surjective. In this case $\pi\beta=0$ and so $\beta$ factors through $M'$. We are thus in the situation $\beta\colon J\to M'\subseteq M$. By induction, we can assume 
\[
\dim J/(J\cap\ann M')+\dim\beta(J\cap\ann M')\leq \dim M'
\]
and we want to show 
\[
\dim J/(J\cap\ann M)+\dim\beta(J\cap\ann M)\leq \dim M.
\]
This is precisely the implication of inequalities in the hypothesis of Proposition \ref{prop:conjecture inductive approach}.
\end{proof}

\bibliographystyle{alpha}
\bibliography{Gerstenhaber}

\begin{thebibliography}{\v{S}12}

\bibitem[Ber13]{Bergman}
George~M. Bergman.
\newblock Commuting matrices, and modules over artinian local rings.
\newblock arXiv:1309.0053, 2013.

\bibitem[BH90]{BH}
Jos\'{e} Barr\'{i}a and P.~R. Halmos.
\newblock Vector bases for two commuting matrices.
\newblock {\em Linear and Multilinear Algebra}, 27(3):147--157, 1990.

\bibitem[BH98]{CM-rings}
W.~Bruns and J.~Herzog.
\newblock {\em Cohen-Macaulay Rings}.
\newblock Cambridge, England: Cambridge University Press, 1998.

\bibitem[Eis95]{Ei}
David Eisenbud.
\newblock {\em Commutative algebra, with a view toward algebraic geometry},
  volume 150 of {\em Graduate Texts in Mathematics}.
\newblock Springer-Verlag, New York, 1995.

\bibitem[Eis05]{Ei2}
David Eisenbud.
\newblock {\em The geometry of syzygies}, volume 229 of {\em Graduate Texts in
  Mathematics}.
\newblock Springer-Verlag, New York, 2005.
\newblock A second course in commutative algebra and algebraic geometry.

\bibitem[Ger61]{Gerstenhaber}
Murray Gerstenhaber.
\newblock On dominance and varieties of commuting matrices.
\newblock {\em Ann. of Math. (2)}, 73:324--348, 1961.

\bibitem[GS]{M2}
Daniel~R. Grayson and Michael~E. Stillman.
\newblock Macaulay2, a software system for research in algebraic geometry.
\newblock Available at \url{http://www.math.uiuc.edu/Macaulay2/}.

\bibitem[Gur92]{Guralnick}
Robert~M. Guralnick.
\newblock A note on commuting pairs of matrices.
\newblock {\em Linear and Multilinear Algebra}, 31(1-4):71--75, 1992.

\bibitem[HE71]{HochsterEagon}
M.~Hochster and John~A. Eagon.
\newblock Cohen-{M}acaulay rings, invariant theory, and the generic perfection
  of determinantal loci.
\newblock {\em Amer. J. Math.}, 93:1020--1058, 1971.

\bibitem[HO01]{HolOmla}
John Holbrook and Matja\v{z} Omladi\v{c}.
\newblock Approximating commuting operators.
\newblock {\em Linear Algebra Appl.}, 327(1-3):131--149, 2001.

\bibitem[HO15]{HolbrookOmeara}
J.~Holbrook and K.~C. O'Meara.
\newblock Some thoughts on {G}erstenhaber's theorem.
\newblock {\em Linear Algebra Appl.}, 466:267--295, 2015.

\bibitem[LL91]{LL}
Thomas~J. Laffey and Susan Lazarus.
\newblock Two-generated commutative matrix subalgebras.
\newblock {\em Linear Algebra Appl.}, 147:249--273, 1991.

\bibitem[MT55]{MTT}
T.~S. Motzkin and Olga Taussky.
\newblock Pairs of matrices with property {$L$}. {II}.
\newblock {\em Trans. Amer. Math. Soc.}, 80:387--401, 1955.

\bibitem[Set11]{SethurSurvey}
B.~A. Sethuraman.
\newblock The algebra generated by three commuting matrices.
\newblock {\em Math. Newsl.}, 21(2):62--67, 2011.

\bibitem[\v{S}12]{Sivic3}
Klemen \v{S}ivic.
\newblock On varieties of commuting triples {III}.
\newblock {\em Linear Algebra Appl.}, 437(2):393--460, 2012.

\bibitem[Wad90]{Wadsworth}
Adrian~R. Wadsworth.
\newblock The algebra generated by two commuting matrices.
\newblock {\em Linear and Multilinear Algebra}, 27(3):159--162, 1990.

\end{thebibliography}

%
%
%
%
%
%
%
%
%
%
%


\end{document}